\documentclass[reqno,12pt]{amsart}
\usepackage{amscd}

\usepackage{mathrsfs}
\usepackage{indentfirst,latexsym,bm,amsmath,amssymb,amsthm,amscd}
\usepackage{diagrams}
\usepackage[all]{xy}
\newtheorem{theorem}{Theorem}[section]
\newtheorem{lemma}[theorem]{Lemma}

\theoremstyle{definition}

\newtheorem{remark}[theorem]{Remark}

\numberwithin{equation}{section}

\begin{document}

\title[canonical maps of nonsingular threefolds]
{On the canonical maps of nonsingular threefolds of general type}

\author[Rong Du]{Rong Du$^{\dag}$}
\address{Department of Mathematics\\
Shanghai Key Laboratory of PMMP\\
East China Normal University\\
Rm. 312, Math. Bldg, No. 500, Dongchuan Road\\
Shanghai, 200241, P. R. China} \email{rdu@math.ecnu.edu.cn}

\thanks{$^{\dag}$ The Research Sponsored by the National Natural Science Foundation of China (Grant No. 11471116, 11531007) and Science and Technology Commission of Shanghai Municipality (Grant No. 13dz2260400).}

 \maketitle

\begin{center}
{\small{Dedicate to Professor Ngaiming Mok on the occasion of his
$60^{\text{th}}$ Birthday.}}
\end{center}

 \begin{abstract}{Let $S$ be a nonsingular minimal complex projective surface of general type and the canonical map of $S$ is generically finite. Beauville showed that the geometric genus of the image of the canonical map is vanishing or equals the geometric genus of $S$ and discussed the canonical degrees for these two cases. We generalize his results to nonsingular minimal complex projective threefolds. }\end{abstract}
\vspace{1cm}
\textbf{AMS subject classifications:} 14J30, 14E20.
\vspace{1cm}

\textbf{Key words:} projective threefold $\cdot$ general type $\cdot$ canonical map $\cdot$
canonical degrees

\section{Introduction}\label{secintro}
The study of the canonical maps of $n$ dimensional projective varieties of general type is one of the central problems in algebraic geometry. For $n=2$, Beauville (\cite{Bea}) proved that the geometric genus of the image $\Sigma$ of the canonical map of a nonsingular complex projective surface $S$ of general type is vanishing or equals the geometric genus of $S$.  Moreover, he showed that the canonical degree is less than or equal to $36$ if $p_g(\Sigma)=0$ and $9$ if $p_g(\Sigma)=p_g(S)$.

For$n=3$, M. Chen studied the canonical map of fiber type (\cite{Ch1}, \cite{C-H}) and posted an open problem in \cite{Ch1} as follows: Let $X$ be a Gorenstein minimal projective $3$-fold with at worst locally factorial terminal singularities. Suppose that the canonical map is generically finite onto its image. Is the generic degree of the canonical map universally upper bounded? Hacon gave a positive answer to Chen's problem. More precisely, he showed that the canonical degree is at most $576$ and the statement is wrong without the Gorenstein condition. Recently, Y. Gao and the author improved Hacon's upper bound to $360$ and the equality holds if and only if $p_g(X)=4$, $q(X)=2$, $\chi(\omega_X)=5$, $K_X^3=360$ and $|K_X|$ is base point free (\cite{D-G1}).

For $n\ge 4$, the situation is totally different. The reason is that, in case of $n<4$, Miyaoka-Yau inequality play a vital role in the proof while Miyaoka-Yau inequality is not effective enough to control $K_X^n$ for $n\ge 4$. In \cite{D-G3}, the authors consider abellian canonical $n$-folds (see \cite{D-G2} for $n=2$ and \cite{D-G1} for $n=3$) such that the canonical degrees are universally upper bounded.

First, we list the following theorem due to  Beauville.

\begin{theorem}\emph{(\cite{Bea})}\label{maint}
Let $X$ be a nonsingular minimal complex projective threefold of
general type. Suppose that the canonical map $\phi_X:
X\dashrightarrow \mathbb{P}^{p_g(X)-1}$ is generically finite. Denote $\Sigma:=\phi_X(X)$, then either:
\begin{itemize}
\item[(1)] $p_g(\Sigma) = 0$, or\\
\item[(2)] $p_g(\Sigma) = p_g(X)$, where the canonical map of $\Sigma$ is of degree $1$.
\end{itemize}
\end{theorem}

In \cite{Sup1}, P. Supino produced a bound for the admissible degree of the canonical map of a
threefold of general type $X$ with the condition $h^2(\mathscr{O}_X)-h^1(\mathscr{O}_X)\ge 2$.
In this paper, we generalize Beauville's other results to minimal nonsingular complex projective threefolds without the above condition.

\begin{theorem}\label{maint2}
Let $X$ be a nonsingular minimal complex projective threefold of
general type. Suppose that the canonical map $\phi_X: X\dashrightarrow \mathbb{P}^{p_g(X)-1}$ is generically finite. Denote $\Sigma:=\phi_X(X)$ and $d:=deg\ \phi_X$.
\begin{itemize}
\item[(1)] If $p_g(\Sigma) = 0$ and $K_M$ is nef, where $M$ is a nonsingular model of $\Sigma$, then $d\le 108$.\\
\item[(2)] If $p_g(\Sigma) = p_g(X)$, where the canonical map of $\Sigma$ is of degree $1$, then $d\le 86$. Moreover, if $p_g(X)\ge 111$, then $d\le 32$.
\end{itemize}
\end{theorem}

\begin{remark}
P. Supino also constructed some examples of threefolds of general type with canonical degree $3$ (\cite{Sup2}) and  $4$, $5$, $6$ with G. Casnati (\cite{C-S}). later, Cai (\cite{Cai}) constructed some examples of threefolds with canonical degrees $32$, $64$ and $72$. But his example of degree $72$ depends on the existence of the surface of general type with canonical degree $36$, which was first discovered by Yeung (cf. \cite{Yeung}, \cite{L-Y}). Each example is a threefold which can be decomposed as Cartesian product of a surface and a curve. In \cite{D-G1}, Gao and the author using abelian cover to construct examples with canonical degree $2$, $4$, $8$, $16$, $32$.
\end{remark}

%

\section{canonical map of threefolds of general type}
Let $S$ be a nonsingular minimal surface of general type with geometric genus $p_g(S)\ge 3$.
Denote by $\phi_S: S\dashrightarrow \mathbb{P}^{p_g(S)-1}$ the canonical map and let $d:= \text{deg}(\phi_S)$. The following
Beauville's result is well-known.

\begin{theorem}\emph{(\cite{Bea})}\label{B}
If the canonical image $F:=\phi_S(S)$ is a surface, then either:
\begin{itemize}
\item[(1)] $p_g(F) = 0$, or\\
\item[(2)] $p_g(F) = p_g(S)$, where $F$ is a canonical surface.
\end{itemize}
Moreover, in case \emph{(1)},
\begin{itemize}
\item[(i)] $d\le 36$ and $d=36$ if and only if $p_g(S)=3$, $q(S)=0$, $K_S^2=36$ and $|K_S|$ base point free,\\
\item[(ii)]$d\le 9$ if $\chi(\mathscr{O})\ge 31$,\\
\item[(iii)] $d\le 11$ if $F$ is not a ruled surface;
\end{itemize}
and in case \emph{(2)},
\begin{itemize}
\item[(i)] $d\le 9$ and $d=9$ if and only if $p_g(S)=4$, $q(S)=0$, $K_S^2=45$, $|K_S|$ base point free and $F$ is a surface of degree $5$ in $\mathbb{P}^3$,\\
\item[(ii)]$d\le 3$ if $\chi(\mathscr{O})\ge 14$.
\end{itemize}
\end{theorem}
\begin{remark}
The existence of a surface of general type with $p_g(S)=3$, $q(S)=0$, $K_S^2=36$ and $|K_S|$ base point free was first proved by Yeung (cf. \cite{Yeung}, \cite{L-Y}).
\end{remark}

The following theorem is due to Beauville. We mimic Beauville's proof to nonsingular minimal complex projective threefolds as follows in order to keep this note self-contained.

Proof of \textbf{ Theorem \ref{maint}}:
\begin{proof}
Let $|K_X|=|S|+F$ such $S$ is a moving part and $F$ is the fixed part of $|K_X|$.
Consider the following diagram
\[\xymatrix{
\tilde{X} \ar[r]^{\pi}\ar[d]_{\sigma_2} &Y\ar[d]^{\varepsilon_2}\\
\bar{X} \ar@{-->}[ur]^{\bar{\phi}}\ar[d]_{\sigma_1}\ar[r]^{f}\ar[dr]& V \ar[d]^{\varepsilon_1}\\
X \ar@{-->}[r] & \Sigma
},\] where $\sigma_1$ is the elimination of the indeterminacy of $\phi_X$, $\sigma_2$ is the elimination of the indeterminacy of $\bar{\phi}$, $V$ is the minimal model of $\Sigma$ and $Y$ is a resolution of $V$. Suppose that $\varepsilon=\varepsilon_1\circ\varepsilon_2$ and $\sigma=\sigma_1\circ\sigma_2$.
Let $|\sigma^*S|=|S'|+\sum_{i=1}^r a_iE_i$ such that $|S'|$ is base point free and
\[\tilde{\phi}=\varepsilon\circ\pi=\phi_{|S'|}: \tilde{X}\rightarrow \Sigma\subseteq \mathbb{P}^{p_g{(X)}-1}.\]
Denote  $\mathscr{O}_{Y}(H)=\varepsilon^*\mathscr{O}_{\Sigma}(1)$, then $S'=\pi^*\circ\varepsilon^*\mathscr{O}_\Sigma(1)=\pi^*(\mathscr{O}_{Y}(H))$.

\begin{equation}
\begin{split}
K_{\tilde{X}}&=\sigma^*K_X+\sum_{i=1}^r b_iE_i\\
&=\sigma^*S+\sigma^*F+\sum_{i=1}^rb_iE_i\\
&=S'+Z,
 \end{split}
  \end{equation}
where $b_i=1$ or $2$, $Z=\sum_{i=1}^r (a_i+b_i)E_i+\sigma^*F$. Then $|K_{\tilde{X}}|=|S'|+Z$.
So
\begin{equation}
\begin{split}
p_g(X)&=h^0(\mathscr{O}_{\tilde{X}}(K_{\tilde{X}}))=h^0(\mathscr{O}_{\tilde{X}}(S'))=h^0(\pi^*\mathscr{O}_{Y}(H))\\
&\ge h^0(\mathscr{O}_{Y}(H))\ge p_g(X).
 \end{split}
  \end{equation}

Therefore $h^0(\mathscr{O}_{\tilde{X}}(S'))=h^0(\mathscr{O}_{Y}(H))=p_g(X)$ and $|S'|=\pi^*|H|$ from which we have $|K_{\tilde{X}}|=\pi^*|H|+Z$.

Suppose that $p_g(Y)=p_g(\Sigma)\neq 0$, then there is a nonzero holomorphic $3$-form $\omega\in H^0{(Y, \mathscr{O}_Y(K_Y))}$.

Next we state that there exists $Q\in |H|$ such that div$(\omega)\ge Q$.
In fact by Hurwitz formula
\[\text{div}(\pi^*(\omega))=\pi^*\text{div}(\omega)+\sum_i (e_i-1)S_i+\sum_{j}r_jE_j,\]
where $S_i$'s are branch loci, $e_i$'s are the ramification indices and $E_j$'s are contracted to points or curves by $\pi$.
Since div$(\pi^*(\omega))\in |K_{\tilde{X}}|=\pi^*|H|+Z$, there exists $Q\in |H|$ such that
\begin{equation}\label{pull Q}
\pi^*Q+Z=\pi^*\text{div}(\omega)+\sum_i(e_i-1)S_i+\sum_j r_jE_j
\end{equation}

Let $Q=\sum h_{\Gamma}\Gamma$ and div$(\omega)=\sum k_{\Gamma}\Gamma$.
Suppose that $\pi^*\Gamma=\sum_t \Gamma_t+\sum r_j'E_j$, where $\Gamma=\pi(\Gamma_t)$.
If $\pi^*\Gamma\nsupseteq S_i$,
since $$\pi^*(Q)=\pi^*(\sum h_\Gamma\Gamma)=\sum h_\Gamma\pi^*\Gamma$$ and
\[\pi^*(\text{div}(\omega))=\sum k_\Gamma\pi^*\Gamma+\sum_i (e_i-1)S_i+\sum_{j}r_jE_j,\]
we have $h_\Gamma\le k_\Gamma$ by comparing the coefficients before $\Gamma_t$ in the both sides of equality (\ref{pull Q}).
If $\Gamma=\pi(S_i)$, suppose that $\pi^*\Gamma=e_iS_i+\sum_{j}r_j'E_j$.
Similarly, we have $h_\Gamma e_i\le k_\Gamma e_i+(e_i-1)$ by comparing the coefficients before $\Gamma_t$ in the both sides of equality (\ref{pull Q}). So $h_\Gamma\le k_\Gamma$.
Therefore we have $K_Y=\text{div}(\omega)\ge Q$. So $h^0(\mathscr{O}_Y(K_Y))\ge h^0(\mathscr{O}_Y(Q))=p_g(X)$. On the other hand, $h^0(\mathscr{O}_Y(K_Y))\le h^0(\mathscr{O}_{\tilde{X}}(K_{\tilde{X}}))=p_g(X)$. So
\[p_g(X)=p_g(Y)=p_g(\Sigma)=h^0(\mathscr{O}_Y(Q)).\]
So $|K_Y|=|Q|+Z'$ and $\phi_{|K_Y|}=\phi_{|Q|}=\varepsilon$ is birational.
\end{proof}

\begin{lemma}\label{degimag}
Let $\Sigma\subseteq \mathbb{P}^n$ be a non-degenerated projective variety of dimension $3$. Take a resolution $\varepsilon: M\rightarrow \Sigma$. If $K_M$ is nef, then deg$\Sigma\ge 2n-4$.
\end{lemma}
\begin{proof}
Let $\mathscr{O}_M(H)=\varepsilon^*(\mathscr{O}_{\Sigma}(1))$. So $|H|$ is base point free. Take an element $S\in |H|$ such that $S$ is irreducible and nonsingular surface. Then $S^3=\text{deg}\Sigma>0$. Notice that $|H|$ restrict to $S$, $|H|_S$, is also base point free, so we can take  $C\in |H|_S$ a nonsingular irreducible curve such that $\mathscr{O}_S(C):=\mathscr{O}_M(H)_S$. Then $C^2=S^3=\text{deg}\Sigma$.
Consider the exact sequence \[0\rightarrow\mathscr{O}_M\rightarrow\mathscr{O}_M(S)\rightarrow\mathscr{O}_S(S)\rightarrow 0.\]
Take cohomology of the exact sequence, we can get $$h^0(\mathscr{O}_S(C))=h^0(\mathscr{O}_S(S))\ge h^0(\mathscr{O}_M(S))-1\ge (n+1)-1=n.$$
From the exact sequence  \[0\rightarrow\mathscr{O}_S\rightarrow\mathscr{O}_S(C)\rightarrow\mathscr{O}_C(C)\rightarrow 0,\] we can also have $h^0(\mathscr{O}_S(C))\le h^0(\mathscr{O}_C(C))+1$.
If $h^1(\mathscr{O}_C(C))\neq 0$, then by Clifford theorem, we have \[n-1\le h^0(\mathscr{O}_C(C))\le \frac{1}{2}C^2+1.\] So deg$\Sigma\ge 2n-4$.
If $h^1(\mathscr{O}_C(C))=0$, by Riemann-Roch theorem, we have $h^0(\mathscr{O}_C(C))=C^2-g(C)+1$.  So deg$\Sigma\ge n-2+g(C)$. On the other hand,
\begin{equation}
\begin{split}
2g(C)-2&=C^2+CK_S\\
&=C^2+C(K_M+S)|_S\\
&=2C^2+K_MS^2\\
&\ge 2C^2.
 \end{split}
  \end{equation}
  So $g(C)\ge \text{deg}\Sigma+1$, a contradiction. Therefore deg$\Sigma\ge 2n-4$.
\end{proof}

With the notations as above, in \cite{D-G1}, Gao and the author showed that the canonical degree $d\le360$ and the equality holds if and only if $p_g(X)=4$, $q(X)=2$, $\chi(\omega_X)=5$, $K_X^3=360$ and $|K_X|$ is base point free,  which is the generalization of Beauville's result of case (1) (i). In \cite{Cai}, Cai showed that if $p_g(X)\ge 105412$, then $d\le 72$ which is the generalization of Beauville's result of case (1) (ii). Next, we are going to generalize Beaville's other results in Theorem \ref{B}.

Proof of \textbf{Theorem \ref{maint2}}:
\begin{proof}
By the condition, one has that $p_g(X)\ge 5$.

For the case (1), by  Lemma \ref{degimag} and Miyaoka-Yau inequality (\cite{Mi}), we have
\[d(2p_g(X)-6)\le d\cdot \text{deg}\Sigma \le K_X^3\le 72\chi(\omega_X).\]
If we can show $\chi(\omega_X)\le p_g(X)+1$, then
\begin{equation}\label{star}
d\le 36\frac{\chi(\omega_X)}{p_g(X)-3}\le 36\frac{p_g(X)+1}{p_g(X)-3}=36(1+\frac{4}{p_g(X)-3})\le108.
\end{equation}
If $q(X)\le2$, then $\chi(\omega_X)\le p_g(X)+q(X)-1\le p_g(X)+1$.

Now we can assume hereafter that $q(X)\ge 3$. Consider the Albanese map $alb_X$ of $X$ and the Stein factorization $f$ of $alb_X$ as follows:

\[\xymatrix{
X \ar[r]^{f}\ar[dr]_{alb_X}
& Y \ar[d]\\
& Alb(X)
}.\]

By Hacon's argument (see the proof of \cite{Ha}, Theorem 1.1), one has
\begin{enumerate}
\item [(1)] $\chi(\omega_X)\le p_g(X)$, if dim$Y\ge 2$;
\item [(2)] $\chi(\omega_X)\le p_g(X)+\chi(\omega_Y)$ and $\chi(\omega_Y)p_g(F)\le p_g(X)$, where $F$ is the general fiber of $f$, if dim$Y=1$.
\end{enumerate}
Hence if dim$Y\ge 2$, by (\ref{star}), the statement holds. More precisely,
\[d\le 36\frac{p_g(X)}{p_g(X)-3}\le 90.\]

We only need to consider dim$Y=1$. By the argument in the main theorem of \cite{D-G1}, we can get that $p_g(F)\ge$ dim$X=3$ and $p_g(X)\ge 6$.

Therefore
\[d\le 36\frac{\chi(\omega_X)}{p_g(X)-3}\le 36\frac{p_g(X)+\chi(\omega_Y)}{p_g(X)-3}\le 36(1+\frac{1}{p_g(F)})\frac{p_g(X)}{p_g(X)-3}\le96.\]

For the case (2), by  M. Chen's result (\cite{Ch2}) we have deg$\Sigma\ge 3p_g(X)-10$. So
\[d(3p_g(X)-10)\le K_X^3\le 72\chi(\omega_X).\]
Using the same arguments as above, we have, if $q(X)\le2$,
\begin{equation}
d\le 72\frac{p_g(X)+1}{3p_g(X)-10}\le86.
\end{equation}
If $q(X)\ge3$ and
\begin{itemize}
\item[(1)] if dim$Y\ge 2$, then $$d\le 72\frac{p_g(X)}{3p_g(X)-10}\le 72;$$\\
\item[(2)] if dim$Y\ge 1$, then $$d\le 72(1+\frac{1}{p_g(F)})\frac{p_g(X)}{3p_g(X)-10}\le72.$$
\end{itemize}

Moreover, suppose that $p_g(X)\ge 111$.

If $q(X)\le2$, then
\begin{equation}
d\le 72\frac{p_g(X)+1}{3p_g(X)-10}\le24.
\end{equation}

If $q(X)\ge3$ and
\begin{itemize}
\item[(1)] if dim$Y\ge 2$, then $$d\le 72\frac{p_g(X)}{3p_g(X)-10}\le 24;$$\\
\item[(2)] if dim$Y\ge 1$, then $$d\le 72(1+\frac{1}{p_g(F)})\frac{p_g(X)}{3p_g(X)-10}\le32.$$
\end{itemize}
So Theorem \ref{maint2} has been proved.
\end{proof}

\section*{Acknowledgements}
The author would like to thank N. Mok for supporting his researches when he was in the University of Hong Kong, to P. Supino for pointing out their several interested papers and to Ching-Jui Lai for spotting some errors in the earlier drafts. The author also would like to thank S.-L. Tan for his unpublished suggestive and elegant book on algebraic surface.


\end{document}